\documentclass[reqno]{amsart}
\usepackage{amssymb,amsbsy,latexsym,amssymb,mathrsfs,mathbbol,mathtools,mathabx,polynom}
\usepackage[utf8]{inputenc}
\usepackage[IL2]{fontenc}
\usepackage{manfnt,pifont}
\usepackage{epsfig}
\usepackage{time}
\usepackage{clock}
\usepackage{units} 
\usepackage{enumerate}
\usepackage{chemfig}

\usepackage{style}

\usepackage[all,2cell]{xy}
\UseTwocells
\newdir{ >}{{}*!/-3pt/\dir{>}}
\SelectTips{cm}{}

\usepackage{tikz}
\usetikzlibrary{calc,through,backgrounds,shapes.geometric,calendar,matrix,arrows}
\usepackage{tikz-cd} 

\usepackage[colorlinks,pageanchor=true]{hyperref}
\hypersetup{pdfauthor=Matej Dostal,pdftitle=Birkhoff,
           pdfstartview=FitH}
\begin{document}

\title[A two-dimensional Birkhoff's theorem]
      {A two-dimensional Birkhoff's theorem}
      
\author{Mat\v{e}j Dost\'{a}l}
\address{Department of Mathematics, Faculty of Electrical Engineering, Czech Technical University
         in Prague, Czech Republic}
\email{dostamat@math.feld.cvut.cz}

\thanks{Mat\v{e}j Dost\'{a}l acknowledges the support
        by the grant SGS14/186/OHK3/3T/13 of \v{C}VUT}

\date{31 August 2015}

\begin{abstract}
Birkhoff's variety theorem from universal algebra characterises equational subcategories of varieties. We give an analogue of Birkhoff's theorem in the setting of enrichment in categories. For a suitable notion of an equational subcategory we characterise these subcategories by their closure properties in the ambient algebraic category.
\end{abstract}

\maketitle

\section{Introduction}

In this paper we will state and give a proof of a 2-dimensional analogue of the Birkhoff theorem from universal algebra. Recall that in the ordinary setting, Birkhoff's theorem characterises equationally defined subcategories of algebraic categories. Consider a category $\Alg(T)$ of algebras for a strongly finitary monad $T$. (Note that a monad is strongly finitary if its underlying functor is strongly finitary, i.~e., if it preserves sifted colimits.) A full subcategory $\A$ of $\Alg(T)$ is said to be \emph{equationally defined} if it is (equivalent to) the category $\Alg(T^\prime)$ of algebras for a strongly finitary monad $T^\prime$, where $T^\prime$ is constructed by ``adding new equations'' to the theory $T$. More precisely, we ask $T^\prime$ to be a quotient of $T$, meaning that there is a monad morphism $e: T \to T^\prime$ that is moreover a regular epimorphism. The resulting algebraic functor
$$
\Alg(e): \Alg(T^\prime) \to \Alg(T)
$$
then exhibits $\Alg(T^\prime)$ as an equationally defined full subcategory of $\Alg(T)$.
Every such subcategory
$
\Alg(T^\prime) \to \Alg(T)
$
has nice closure properties with respect to to the inclusion into $\Alg(T)$. The content of Birkhoff's theorem is that equationally defined subcategories can moreover be characterised by these closure properties. In essence, this theorem holds since algebraic categories are well-behaved with respect to quotients (regular epis) -- they are \emph{exact categories}~\cite{adamek+rosicky+vitale}.

Taking inspiration from the ordinary case, we want to give a characterisation of equationally defined subcategories of algebraic categories in the enriched setting. Namely, we shall mainly work with categories enriched in the symmetric monoidal closed category $\V = \Cat$ and we will accordingly use the enriched notions of a functor, natural transformation, etc. 

Analogously to the ordinary case, in defining the notion of an equationally defined subcategory of $\Alg(T)$ the idea is again to consider ``quotients'' $e: T \to T^\prime$ of strongly finitary 2-monads. Any subcategory $\Alg(e): \Alg(T^\prime) \to \Alg(T)$ exhibited by a quotient $e: T \to T^\prime$ is equationally defined in $\Alg(T)$.

Unlike in the $\V = \Set$ case, it is not immediately clear that some well-behaved notion of a quotient of strongly finitary 2-monads should exist. In $\V = \Set$, the quotients come as the epi part of the (regular epi, mono) factorisation system, and they are computed as certain colimits, the coequalisers. The solution in $\V = \Cat$ is to mimic this approach. Thus we should study factorisation systems on $\Cat$ (and the respective notions of a quotient), and find out which factorisation systems ``lift up'' from the category $\Cat$ to $\Cat$-enriched algebraic categories. That is, we want to find factorisation systems that render the algebraic categories over $\Cat$ \emph{exact} in some suitably generalised sense. This would allow us to talk about quotients of strongly finitary 2-monads while preserving the good behaviour of quotients as in $\Cat$. 

Recent advances in the theory of 2-dimensional exactness (see~\cite{bourke+garner-exactness}) show that there are at least three notions of a quotient coming from three factorisation systems $(\E,\M)$ on $\Cat$, for which algebraic categories over $\Cat$ are exact:
\begin{enumerate}
\item (surjective on objects, injective on objects and fully faithful),
\item (bijective on objects, fully faithful),
\item (bijective on objects and full, faithful).
\end{enumerate}
(For the $\E$ parts of the above systems, we will use the standard abbreviations, namely s.o.~for surjective on objects, b.o.~for bijective on objects, and~b.o.~full for bijective on objects and full.)
One of the results of~\cite{bourke+garner-exactness} states that the 2-category $\Mndsf(\Cat)$ of strongly finitary 2-monads over $\Cat$ is \emph{exact} in the sense of~\cite{bourke+garner-exactness} w.r.t.\ all the three factorisation systems above.

We focus on the factorisation system (b.o.~full, faithful). Unlike the other two systems, it corresponds to a meaningful notion of an equationally defined subcategory, and it allows to prove the 2-dimensional Birkhoff theorem by arguments very similar to those contained in the proof of the ordinary Birkhoff theorem. For this factorisation system, the exactness of $\Mndsf(\Cat)$ implies that a monad morphism $e: T \to T^\prime$ is a quotient if and only if $e_C: TC \to T^\prime C$ is a b.o.~full functor in $\Cat$ for every category $C$. We shall often use this ``pointwise'' nature of quotient monad morphisms.

The main result of the paper characterises ``equational subcategories'' of algebraic categories as those that are closed under products, quotients, subalgebras and sifted colimits. This is a characterisation precisely in the spirit of the ordinary Birkhoff theorem. In the ordinary setting, only the first three closure properties are demanded, and are dubbed ``HSP'' conditions. However, even in the ordinary case, it was found out that closure under \emph{filtered} colimits was necessary when dealing with the many-sorted case~\cite{arv-many-sorted-hsp}. It is thus not surprising that the additional requirement of closure under \emph{sifted} colimits is needed in the 2-dimensional case: the \emph{finitary} and \emph{strongly finitary} 2-monads no longer coincide in $\Cat$, and we are dealing with the strongly finitary ones. The choice of working with strongly finitary 2-monads is fairly natural, since the 2-category $\Mndsf(\Cat)$ is equivalent to the 2-category $\Law$ of $\Cat$-enriched one-sorted algebraic theories (also dubbed \emph{Lawvere 2-theories})~\cite{lack+rosicky}.

In the final section we conclude with a few remarks on possible future work and on the other two factorisation systems on $\Cat$. These systems are much worse behaved, and thus the corresponding Birkhoff-type theorem would be of a weaker nature.

\section{Factorisations, kernels and quotients in 2-categories}

We shall make heavy use of factorisation systems in discussing and proving the Birkhoff theorem. The study of factorisation systems in general 2-categories is much more involved than in the ordinary case. Following the exposition in~\cite{bourke+garner-exactness}, we first recall the definitions of \emph{enriched} orthogonality in a general $\V$-category for a symmetric monoidal closed base category $\V$. Then we introduce \emph{kernel-quotient systems} that generalise the correspondence between regular epimorphisms and kernels in exact categories, and we use this notion to introduce three well-behaved factorisation systems on $\Cat$.

\begin{definition}
\label{def:ortho}
A morphism $f: A \to B$ in a $\V$-category $\C$ is $\V$-orthogonal to $g: C \to D$ if the diagram
$$
\begin{tikzcd}
\C(B,C) \rar{\C(B,g)} \dar[swap]{\C(f,C)} & \C(B,D) \dar{\C(f,D)} \\
\C(A,C) \rar[swap]{\C(A,g)} & \C(A,D)
\end{tikzcd}
$$
is a pullback in $\V$. Given an object $C$ of $\C$, the morphism $f: A \to B$ is orthogonal to $C$ if $f$ is orthogonal to $\id_C$, i.~e., if the precomposition map
$$
\C(f,C): \C(B,C) \to \C(A,C)
$$
is invertible (i.e., an isomorphism). We denote this fact by $f \perp C$.
\end{definition}

\begin{example}
We examine when two morphisms $f: A \to B$ and $g: C \to D$ are orthogonal in $\C$ for the case of $\V = \Cat$. Firstly, the morphisms have to satisfy the usual diagonal fill-in property
$$
\begin{tikzcd}
A \rar{f} \dar[swap]{x} & B \dar{y} \dlar[swap, dotted]{\exists! d} \\
C \rar[swap]{g} & D
\end{tikzcd}
$$
for every pair $x: A \to C$ and $y: B \to D$ of morphisms in $\C$. Let us denote by $d: A \to D$ the diagonal fill-in for $x$ and $y$, and denote by $d': A \to D$ the diagonal fill-in for $x'$ and $y'$. The second requirement on $f$ and $g$ to be orthogonal is that they satisfy the diagonal \emph{2-cell} property: for every pair $\alpha: x \Rightarrow x'$ and $\beta: y \Rightarrow y'$ of 2-cells such that
$$
\begin{tikzpicture}
    \node (X) at (0,1) {$A$};
    \node (V) at (0,-1) {$C$};
    \node (Y) at (2,-1) {$D$};
    \node (L) at (0,0) {$\Rightarrow$};
    \node (alpha) at (0,0.3) {$\alpha$};
    \draw[->] (X) edge [bend left=30] node [right] {$x'$} (V);
    \draw[->] (X) edge [bend right=30] node [left] {$x\phantom{'}$} (V);
    \draw[->] (V) edge node [below] {$g$} (Y);

\node (eq) at (3,0) {$=$};

\begin{scope}[shift={(6,0)}]
    \node (X) at (0,1) {$B$};
    \node (V) at (0,-1) {$D$};
    \node (Y) at (-2,1) {$A$};
    \node (L) at (0,0) {$\Rightarrow$};
    \node (alpha) at (0,0.3) {$\beta$};
    \draw[->] (X) edge [bend left=30] node [right] {$y'$} (V);
    \draw[->] (X) edge [bend right=30] node [left] {$y\phantom{'}$} (V);
    \draw[->] (Y) edge node [above] {$f$} (X);
\end{scope}
\end{tikzpicture}
$$
there has to exist a \emph{unique} 2-cell $\delta: d \Rightarrow d'$ such that the equalities
$$
\begin{tikzpicture}
    \node (X) at (0,1) {$A$};
    \node (V) at (0,-1) {$C$};
    \node (L) at (0,0) {$\Rightarrow$};
    \node (alpha) at (0,0.3) {$\alpha$};
    \draw[->] (X) edge [bend left=30] node [right] {$x'$} (V);
    \draw[->] (X) edge [bend right=30] node [left] {$x\phantom{'}$} (V);

\node (eq) at (1.5,0) {$=$};

\begin{scope}[shift={(4,0)}]
    \node (X) at (0,1) {$B$};
    \node (V) at (-2,-1) {$C$};
    \node (Y) at (-2,1) {$A$};
    \node (L) at (-1,0) {$\Rightarrow$};
    \node (alpha) at (-1,0.3) {$\delta$};
    \draw[->] (X) edge [bend left=30] node [right] {$d'$} (V);
    \draw[->] (X) edge [bend right=30] node [left] {$d\phantom{'}$} (V);
    \draw[->] (Y) edge node [above] {$f$} (X);
\end{scope}
\end{tikzpicture}
\qquad \qquad
\begin{tikzpicture}
    \node (X) at (0,1) {$B$};
    \node (V) at (0,-1) {$D$};
    \node (L) at (0,0) {$\Rightarrow$};
    \node (alpha) at (0,0.3) {$\beta$};
    \draw[->] (X) edge [bend left=30] node [right] {$y'$} (V);
    \draw[->] (X) edge [bend right=30] node [left] {$y\phantom{'}$} (V);

\node (eq) at (1.5,0) {$=$};

\begin{scope}[shift={(4,0)}]
    \node (X) at (0,1) {$B$};
    \node (V) at (-2,-1) {$C$};
    \node (Y) at (0,-1) {$D$};
    \node (L) at (-1,0) {$\Rightarrow$};
    \node (alpha) at (-1,0.3) {$\delta$};
    \draw[->] (X) edge [bend left=30] node [right] {$d'$} (V);
    \draw[->] (X) edge [bend right=30] node [left] {$d\phantom{'}$} (V);
    \draw[->] (V) edge node [below] {$g$} (Y);
\end{scope}
\end{tikzpicture}
$$
hold.

Similarly, a morphism $f: A \to B$ in $\C$ is orthogonal to an object $C$ of $\C$ if
\begin{enumerate}
\item for every $g: A \to C$ there exists a unique morphism $h: B \to C$ such that $h \cdot f = g$, and
\item for every 2-cell $\alpha: g \Rightarrow g'$ there exists a unique 2-cell $\beta: h \Rightarrow h'$ such that $\beta * f = \alpha$ holds.
\end{enumerate} 
\end{example}

We now introduce the notion of a \emph{kernel-quotient system}. This notion generalises the notions of a kernel pair and its induced quotient, and allows treating factorisation systems in enriched categories parametric in the choice of the shape of ``kernel data''. Importantly, this approach covers the motivating ordinary (regular epi, mono) factorisation system on $\Set$ as well as the three factorisation systems on $\Cat$ that are mentioned in the introduction and discussed in detail in Example~\ref{ex:kernel-quotient-systems}.

In the following, we will restrict ourselves to $\V$ being a locally finitely presentable category that is closed as a monoidal category in the sense of~\cite{kelly:structures}, as we will need to impose a finitarity condition on the kernel-quotient system.

Let us denote by $\Two$ the free $\V$-category on an arrow $1 \to 0$. We let $\F$ be a finitely presentable $\V$-category containing $\Two$ as a full subcategory. Then there is the obvious inclusion $J: \Two \to \F$ and the inclusion $I: \K \to \F$ of the full subcategory of $\F$ spanned by all objects of $\F$ except $0$. We call the data $(J,I)$ a \emph{kernel-quotient system}. Given a complete and cocomplete $\V$-category $\C$, there is a chain of adjunctions as in the following diagram:
$$
\begin{tikzpicture}
	\node (2C) at (-2,0) {$[\Two,\C]$};
	\node (FC) at (0,0) {$[\F,\C]$};
	\node (KC) at (2,0) {$[\K,\C]$};

	\node[rotate=-90] (adj) at (-1,0) {$\dashv$};
	\node[rotate=-90] (adj) at (1,0) {$\dashv$};

	\draw[->]
	(2C)
	edge [bend right]
	node [below] {$\Ran{J}{}$}
	(FC);
	
	\draw[->]
	(FC)
	edge [bend right]
	node [below] {$[I,\C]$}
	(KC);
	
	\draw[->]
	(KC)
	edge [bend right]
	node [above] {$\Lan{I}{}$}
	(FC);
	
	\draw[->]
	(FC)
	edge [bend right]
	node [above] {$[J,\C]$}
	(2C); 
\end{tikzpicture}
$$
We denote the composite adjunction by
$$
\begin{tikzpicture}
	\node (2C) at (-1,0) {$[\Two,\C]$};
	\node (KC) at (1,0) {$[\K,\C]$};

	\node[rotate=-90] (adj) at (0,0) {$\dashv$};

	\draw[->]
	(2C)
	edge [bend right]
	node [below] {$K$}
	(KC);
		
	\draw[->]
	(KC)
	edge [bend right]
	node [above] {$Q$}
	(2C);
\end{tikzpicture}
$$
and call it the \emph{kernel-quotient} adjunction for $\F$.

\begin{remark}
In~\cite{bourke+garner-exactness} the authors give a weaker definition of kernel-quotient adjunction to capture the cases where $\C$ is not complete and cocomplete. We do not need to introduce this weaker notion, as the 2-categories $\C$ in our examples always satisfy the completeness conditions.
\end{remark}

\begin{definition}
Given a $\V$-category $\C$ that admits the kernel-quotient adjunction for $\F$, we say that an object $X$ in $[\K,\C]$ is an \emph{$\F$-kernel} if it is in the essential image of $K$. Any arrow $f: A \to B$ in $\C$ is called an \emph{$\F$-quotient map} if it is in the essential image of $Q$, and it is called \emph{$\F$-monic} if the morphism $K(\id_A,f): K(\id_A) \to K(f)$ is an isomorphism.
\end{definition}

\begin{example}
The motivating example of a kernel-quotient adjunction in the ordinary setting ($\V = \Set$) is given by taking  the category $\F$ to be of the shape
$$
\begin{tikzcd}
2 \rar[bend left] \rar[bend right] & 1 \rar & 0.
\end{tikzcd}
$$
Here the adjunction $Q \dashv K$ acts as follows. The functor $Q$ sends a parallel pair $X = (f,g)$ to a coequaliser $QX$ of the parallel pair $(f,g)$. A morphism $f: A \to B$ in $\C$, thus an object in $[\Two,\C]$, is sent by $K$ to the kernel pair $Kf = (k_1, k_2)$ of $f$. The $\F$-monic morphisms are precisely the monomorphisms in this example.
\end{example}

The kernel-quotient system in the previous example allows factoring every morphism in $\C$ as a regular epimorphism followed by a (not necessarily monomorphic) morphism. Since both the functors $I: \K \to \F$ and $J: \Two \to \F$ are injective on objects and fully faithful, the functors $\Ran{J}{}$ and $\Lan{I}{}$ can \emph{always} be taken as strict sections of the functors $[J,\C]$ and $[I,\C]$, respectively. Then the kernel-quotient adjunction $Q \dashv K$ may be taken to commute with the evaluation functors $[\Two,\C] \to \C$ and $[\K,\C] \to \C$ that evaluate at the object $1$. This results in the counit $\eps$ of $Q \dashv K$ having the following form for all objects $f$ in $[\Two,\C]$:
$$
\begin{tikzcd}
A \rar{\id_A} \dar[swap]{QKf} & A \dar{f} \\
C \rar[swap]{\eps_f} & B
\end{tikzcd}
$$
Thus we have a factorisation
$$
f = \eps_F \cdot QKf,
$$
and $QKf$ is a regular epi, being a coequaliser of the parallel pair $Kf$. If the morphism $\eps_f$ is a mono for every $f$, we obtain a factorisation system (regular epi, mono) on $\C$. More generally, we say that \emph{$\F$-kernel-quotient factorisations in $\C$ converge immediately} whenever $\eps_f$ is always $\F$-monic. Whenever $\F$-kernel-quotient factorisations converge immediately in $\C$, we obtain a factorisation system ($\F$-quotient, $\F$-monic) on $\C$ (by Proposition~4 of~\cite{bourke+garner-exactness}).

\begin{remark}
The $\F$-quotient and $\F$-monic maps in $\C$ satisfy many expected properties:
\begin{enumerate}
\item both classes are closed under composition and identities,
\item the $\F$-quotients are closed under pushouts and under colimits in $[\Two,\C]$,
\item the $\F$-monic maps are closed under pullbacks and under limits in $[\Two,\C]$,
\item any finitely-continuous $\V$-functor preserves $\F$-monics,
\item any left adjoint $\V$-functor preserves $\F$-quotients, and
\item any $\F$-quotient $f: A \to B$ is $\V$-orthogonal to any $\F$-monic $g: C \to D$.
\end{enumerate}
These facts are proved in~Proposition~2 of~\cite{bourke+garner-exactness}.
\end{remark}

\begin{example}
\label{ex:kernel-quotient-systems}
In this example we present various kernel-quotient systems in the case of enrichment in $\Cat$. This example is essentialy the contents of Section~5.1 of~\cite{bourke+garner-exactness}.
\begin{enumerate}
\item Given a 2-category $\F_\bof$ generated by
$$
\begin{tikzpicture}
    \node (E) at (-1,0) {$2$};
    \node (V) at (1,0) {$1$};
    \node (A) at (3,0) {$0$};
    \node[rotate=-90] (L) at (-0.15,0) {$\Rightarrow$};
    \node[rotate=-90] (R) at (0.15,0) {$\Rightarrow$};
    \node (alpha) at (-0.4,0) {$\alpha$};
    \node (beta) at (0.4,0) {$\beta$};
    \draw[->] (E) edge [bend left=30] node [above] {$s$} (V);
    \draw[->] (E) edge [bend right=30] node [below] {$t$} (V);
    \draw[->] (V) edge node [above] {$w$} (A);
\end{tikzpicture}
$$
subject to the identity $w * \alpha = w * \beta$, we obtain the following kernel-quotient system. Given kernel-data $X$ in $[\K,\C]$, its $\F_\bof$-quotient $QX$ is its \emph{coequifier}. When $\C = \Cat$, kernel-quotient factorisations for $\F_\bof$ converge immediately, and they give rise to the (b.o.~full, faithful) factorisation system on $\Cat$.

\item Let $\F_\bo$ be the 2-category generated by
$$
\begin{tikzcd}
3 \rar[bend left=40]{p} \rar[bend right=40, swap]{q} \rar{m} & 2 \rar[bend left=40]{d} \rar[bend right=40, swap]{c}  & 1 \rar{w} \lar[swap]{i} & 0.
\end{tikzcd}
$$
together with a 2-cell $\theta: wd \Rightarrow wc$, modulo the identities
$$
d \cdot i = c \cdot i = \id_1, \qquad c \cdot p = d \cdot q, \qquad d \cdot m = d \cdot p, \qquad c \cdot m = c \cdot q,
$$
on 1-cells and modulo the 2-cell identities
$$
\theta * i = \id_w, \qquad (\theta * q) \cdot (\theta * p) = \theta * m.
$$
Given some kernel-data $X$ in $[\K,\C]$,  the $\F_\bo$-quotient of $X$ is the \emph{universal codescent cocone} under $X$. When $\C = \Cat$, kernel-quotient factorisations for $\F_\bo$ again converge immediately, and give rise to the (b.o., f.f.) factorisation system on $\Cat$. 

\item The (s.o., i.o.f.f.) factorisation system on $\Cat$ is given by the kernel-quotient system $\F_\so$
$$
\begin{tikzcd}
& 2' \dar{j} & & \\
3 \rar[bend left=40]{p} \rar[bend right=40, swap]{q} \rar{m} & 2 \rar[bend left=40]{d} \rar[bend right=40, swap]{c}  & 1 \rar{w} \lar[swap]{i} & 0
\end{tikzcd}
$$
obtained from the system $\F_\bo$ in (2) by adjoining a new morphism $2' \xrightarrow{j} 2$ subject to equalities
$$
w \cdot d \cdot j = w \cdot c \cdot j, \qquad \theta * j = \id_{wdj}.
$$
An $\F_\so$-quotient of kernel-data $X$ is a codescent cocone universal amongst those cocones for which $\theta * Xj$ is the identity 2-cell.
\end{enumerate}
\end{example}

\begin{remark}
The main focus of~\cite{bourke+garner-exactness} is to study the generalised notions of regularity and exactness, parametric in the choice of a kernel-quotient system $\F$. This yields a theory of $\F$-regularity and $\F$-exactness. We do not need to introduce the theory of $\F$-exactness in detail. In fact, we use the results of~\cite{bourke+garner-exactness} only to ``lift'' the well-behaved factorisation systems of~Example~\ref{ex:kernel-quotient-systems} on $\Cat$ to any algebraic category $\Alg(T)$ for a strongly finitary 2-monad $T$ on $\Cat$.

For our purposes, it is enough to recall the following facts:
\begin{enumerate}
\item
The category $\Cat$ is $\F$-exact (and therefore $\F$-regular) for all the three kernel-quotient systems introduced in~\ref{ex:kernel-quotient-systems}.
\item
Proposition~16 of~\cite{bourke+garner-exactness} states that if kernel-quotient factorisations for $\F$ converge immediately in $\V$, then they converge in every $\F$-exact $\V$-category.
That is, every $\F$-exact $\V$-category then admits an ($\F$-quotient,$\F$-monic) factorisation system.
\item
Let $\F$-quotients be closed under finite products in $\Cat$ for a kernel-quotient system $\F$. Given any small 2-category $\A$ with finite products, the 2-category $\FP[\A,\C]$ of finite-product-preserving functors is $\F$-exact ($\F$-regular), whenever $\C$ is. This follows from Propositions~18 and~56 of~\cite{bourke+garner-exactness}.
\item
In $\Cat$, given any of the kernel-quotient systems introduced in~Example~\ref{ex:kernel-quotient-systems}, the $\F$-quotients are closed under finite products. Consider any $\F$-exact 2-category $\C$ for such a kernel-quotient system $\F$. By Remark~59 of~\cite{bourke+garner-exactness} we get that for any strongly finitary 2-monad $T$ on $\C$, the 2-category $\Alg(T)$ of (strict) $T$-algebras is also $\F$-exact.
\end{enumerate}
\end{remark}

\begin{remark}
\label{rem:sf-monads-are-monadic}
Denote by $N$ the discrete 2-category with natural numbers as objects. An argument from~\cite{lack:monadicity} shows that there is a chain
$$
\begin{tikzpicture}
	\node (mnd) at (0,1.5) {$\Mndsf(\Cat)$};
	\node (catsf) at (0,0) {$[\Catsf,\Cat]$};
	\node (ncat) at (0,-1.5) {$[N,\Cat]$};
	\node (adj1) at (0,0.75) {$\dashv$};
	\node (adj2) at (0,-0.75) {$\dashv$};
		
    \draw[->] (mnd) edge [bend left=30] node [right] {$W$} (catsf);
    \draw[->] (catsf) edge [bend left=30] node [left] {$H$} (mnd);
    \draw[->] (catsf) edge [bend left=30] node [right] {$V$} (ncat);
    \draw[->] (ncat) edge [bend left=30] node [left] {$G$} (catsf);	
\end{tikzpicture}
$$
of adjunctions, with the composite adjunction
$$
\begin{tikzpicture}
	\node (mnd) at (0,1.5) {$\Mndsf(\Cat)$};
	\node (ncat) at (0,0) {$[N,\Cat]$};
	\node (adj1) at (0,0.75) {$\dashv$};
		
    \draw[->] (mnd) edge [bend left=30] node [right] {$U$} (ncat);
    \draw[->] (ncat) edge [bend left=30] node [left] {$F$} (mnd);
\end{tikzpicture}
$$
being monadic. Thus $\Mndsf(\Cat)$ is equivalent to the 2-category $[N,\Cat]^M$ of Eilenberg-Moore algebras for the 2-monad $M = UF$. The right adjoint $U$ preserves sifted colimits, and therefore $M$ is strongly finitary. The category $[N,\Cat]^M$ is a locally finitely presentable category (in the 2-dimensional sense of~\cite{kelly:structures}), and so it is complete and cocomplete. We can conclude that $\Mndsf(\Cat)$ admits all the three factorisation systems of Example~\ref{ex:kernel-quotient-systems}, being $\F$-exact for the respective kernel-quotient systems.
\end{remark}

\section{The (b.o.~full, faithful) factorisation system}

The present section is devoted to the study of the very special (b.o.~full, faithful) factorisation system on $\Cat$ and on other $\F_\bof$-exact categories. Most importantly, we show first that the $\F_\bof$-quotients (b.o.~full functors in the case of the 2-category $\Cat$) enjoy many convenient properties. These properties will allow us to prove a $\Cat$-enriched Birkhoff theorem in the following section, with the proof being very much in the spirit of the proof for ordinary Birkhoff theorem. We shall also see that quotients of monads correspond to ``equationally defined'' subcategories of algebraic categories. As a slight detour we shall prove a 2-dimensional adjoint functor theorem that will be applied in the proof of Birkhoff theorem. Again, the applicability of the adjoint functor theorem is special to the (b.o.~full, faithful) factorisation system due to its nice behaviour.

Given a b.o.~full $h: C \to A$ in $\Cat$, the functor
$$
\Cat(h,B): \Cat(A,B) \to \Cat(C,B)
$$
is injective on objects and fully faithful for every $B$. The injectivity on objects of $\Cat(h,B)$ correspods to  $h$ being an epimorphism in $\Cat$, faithfulness of $\Cat(h,B)$ states that $h$ is an epimorphism with respect to 2-cells, and fullness of $\Cat(h,B)$ corresponds to a factorisation property of $h$ w.r.t.\ 2-cells. We prove this easy fact in the following lemma.

\begin{lemma}
\label{lem:cancel-2-cells}
Consider the following diagram
$$
\begin{tikzcd}
{} & A \arrow{dr}{f} \arrow[Rightarrow, short]{dd}{\beta} & {} \\
C \urar{h} \drar[swap]{h} & {} & B \\
{} & A \urar[swap]{g} & {}
\end{tikzcd}
$$
in $\Cat$ and let $h$ be bijective on objects and full. Then there is a unique natural transformation 
$$
\begin{tikzpicture}
    \node (E) at (-1,0) {$A$};
    \node (V) at (1,0) {$B$};
    \node[rotate=-90] (L) at (0,0) {$\Rightarrow$};
    \node (alpha) at (0.3,0) {$\alpha$};
    \draw[->] (E) edge [bend left=30] node [above] {$f$} (V);
    \draw[->] (E) edge [bend right=30] node [below] {$g$} (V);
\end{tikzpicture}
$$
such that $\beta = \alpha * h$ holds.
\end{lemma}
\begin{proof}
Denote by
$$
\begin{tikzpicture}
    \node (E) at (-1,0) {$K$};
    \node (V) at (1,0) {$C$};
    \node (A) at (3,0) {$A$};
    \node[rotate=-90] (L) at (-0.15,0) {$\Rightarrow$};
    \node[rotate=-90] (R) at (0.15,0) {$\Rightarrow$};
    \node (alpha) at (-0.4,0) {$\gamma$};
    \node (beta) at (0.4,0) {$\delta$};
    \draw[->] (E) edge [bend left=30] node [above] {$s$} (V);
    \draw[->] (E) edge [bend right=30] node [below] {$t$} (V);
    \draw[->] (V) edge node [above] {$h$} (A);
\end{tikzpicture}
$$
the kernel-quotient pair of $h$. Therefore, since $\Cat$ is $\F_\bof$-exact, $h$ is a coequifier of the diagram
$$
\begin{tikzpicture}
    \node (E) at (-1,0) {$K$};
    \node (V) at (1,0) {$C$};
    \node[rotate=-90] (L) at (-0.15,0) {$\Rightarrow$};
    \node[rotate=-90] (R) at (0.15,0) {$\Rightarrow$};
    \node (alpha) at (-0.4,0) {$\gamma$};
    \node (beta) at (0.4,0) {$\delta$};
    \draw[->] (E) edge [bend left=30] node [above] {$s$} (V);
    \draw[->] (E) edge [bend right=30] node [below] {$t$} (V);
\end{tikzpicture}
$$
Then both $f \cdot h$ and $g \cdot h$ also coequify the above diagram. By the 2-dimensional property of coequifiers we have that the equality
$$
\begin{tikzpicture}
    \node (X) at (0,1) {$C$};
    \node (V) at (0,-1) {$B$};
    \node (L) at (0,0) {$\Rightarrow$};
    \node (alpha) at (0,0.3) {$\beta$};
    \draw[->] (X) edge [bend left=30] node [right] {$gh$} (V);
    \draw[->] (X) edge [bend right=30] node [left] {$fh$} (V);

\node (eq) at (3,0) {$=$};

\begin{scope}[shift={(6,0)}]
    \node (X) at (0,1) {$A$};
    \node (V) at (0,-1) {$B$};
    \node (Y) at (-2,1) {$C$};
    \node (L) at (0,0) {$\Rightarrow$};
    \node (alpha) at (0,0.3) {$\alpha$};
    \draw[->] (X) edge [bend left=30] node [right] {$g$} (V);
    \draw[->] (X) edge [bend right=30] node [left] {$f$} (V);
    \draw[->] (Y) edge node [above] {$h$} (X);
\end{scope}
\end{tikzpicture}
$$
holds for a unique 2-cell $\alpha$. From this it immediately follows that $\Cat(h,B)$ is fully faithful.
\end{proof}

Incidentally, the above lemma follows directly from Proposition~3.6 of~\cite{bashir-velebil}, where the authors study properties of \emph{connected} functors. These are exactly the functors $h$ for which the precomposition functor $\Cat(h,B)$ is fully faithful for every $B$.

In fact, it is sufficient that $h$ be surjective on objects for the functor $\Cat(h,B)$ to be faithful for every category $B$.

\begin{lemma}
Let
$$
\begin{tikzpicture}
    \node (E) at (-1,0) {$A$};
    \node (V) at (1,0) {$B$};
    \node[rotate=-90] (L) at (-0.15,0) {$\Rightarrow$};
    \node[rotate=-90] (R) at (0.15,0) {$\Rightarrow$};
    \node (alpha) at (-0.4,0) {$\alpha$};
    \node (beta) at (0.4,0) {$\beta$};
    \draw[->] (E) edge [bend left=30] node [above] {$s$} (V);
    \draw[->] (E) edge [bend right=30] node [below] {$t$} (V);
\end{tikzpicture}
$$
be a diagram in $\Cat$. Given a functor $h: C \to A$ that is surjective on objects, the equality $\alpha = \beta$ holds if and only if $\alpha * h = \beta * h$ holds.
\end{lemma}

\begin{proof}
The direction $(\Rightarrow)$ is immediate. For the direction $(\Leftarrow)$, consider the components
$$
\alpha_a: fa \to ga, \qquad \beta_a : fa \to ga
$$
of $\alpha$ and $\beta$ for a fixed $a$ in $A$. Since $h$ is surjective on objects, we have some $c$ in $C$ with $hc = a$, and so
$$
\alpha_a = \alpha_{hc} = (\alpha * h)_c = (\beta * h)_c = \beta_{hc} = \beta_a,
$$
which was to be proved.
\end{proof}

We will show that a quotient $e: T \twoheadrightarrow T^\prime$ in $\Mndsf(\Cat)$ is b.o.~full pointwise. That is, the functor $e_n: Tn \twoheadrightarrow T^\prime n$ is b.o.~full for every finite discrete category $n$. For this observation we will need the following crucial lemma, stating that coequifiers can be ``made reflexive''.

\begin{lemma}
\label{lem:coeq-refl}
Let $\K$ be a 2-category with finite 2-coproducts. Then, if $e: B \to C$ is a coequifier of the diagram
$$
\begin{tikzpicture}
    \node (E) at (-1,0) {$A$};
    \node (V) at (1,0) {$B$};
    \node[rotate=-90] (L) at (-0.15,0) {$\Rightarrow$};
    \node[rotate=-90] (R) at (0.15,0) {$\Rightarrow$};
    \node (alpha) at (-0.4,0) {$\alpha$};
    \node (beta) at (0.4,0) {$\beta$};
    \draw[->] (E) edge [bend left=30] node [above] {$s$} (V);
    \draw[->] (E) edge [bend right=30] node [below] {$t$} (V);
\end{tikzpicture}
$$
it is also a coequifier of the reflexive diagram
$$
\begin{tikzpicture}
	\node (temp1) at (0,1) {};
	\node (temp2) at (0,-1) {};
    \node (E) at (-1.7,0) {$A + B$};
    \node (V) at (1.7,0) {$B$};
    \node[rotate=-90] (L) at (-0.15,0.3) {$\Rightarrow$};
    \node[rotate=-90] (R) at (0.15,0.3) {$\Rightarrow$};
    \node (alpha) at (-0.6,0.3) {$[\alpha,1]$};
    \node (beta) at (0.6,0.3) {$[\beta,1]$};
    
    \draw[transform canvas={yshift=4ex},->] (E) edge node [above] {$[s,\id_B]$} (V);
    \draw[->] (E) edge  node [below] {$[t,\id_B]$} (V);
    \draw[transform canvas={yshift=-4ex},->] (V) edge node [below] {$i_B$} (E);
\end{tikzpicture}
$$
where $i_B: B \to A + B$ is the injection of $B$ into $A + B$.
\end{lemma}

\begin{proof}
To check that the above diagram is indeed reflexive we need to check whether the equalities
\begin{align*}
[s,\id_B] \cdot i_B = [t,\id_B] \cdot i_B = \id_B \\
[\alpha,1] * i_B = [\beta,1] * i_B = 1
\end{align*}
hold. This follows directly from the universal property of the 2-coproduct $A + B$.

Now any $e: B \to C$ that coequifies the diagram
$$
\begin{tikzpicture}
    \node (E) at (-1,0) {$A$};
    \node (V) at (1,0) {$B$};
    \node[rotate=-90] (L) at (-0.15,0) {$\Rightarrow$};
    \node[rotate=-90] (R) at (0.15,0) {$\Rightarrow$};
    \node (alpha) at (-0.4,0) {$\alpha$};
    \node (beta) at (0.4,0) {$\beta$};
    \draw[->] (E) edge [bend left=30] node [above] {$s$} (V);
    \draw[->] (E) edge [bend right=30] node [below] {$t$} (V);
\end{tikzpicture}
$$
also coequifies the reflexive diagram, as the restrictions imposed on $e: B \to C$ by the reflexivity conditions are void.
\end{proof}

\begin{remark}
\label{rem:sf-preserve-bof}
The above lemma is important because reflexive coequifiers are examples of \emph{sifted} colimits. We see that, under mild cocompleteness conditions, a coequifier $e: B \to C$ of some pair can be perceived as a coequifier of a reflexive pair. In the case of $\Mndsf(\Cat)$ we thus get that a quotient $e: T \twoheadrightarrow T^\prime$ is a coequifier
$$
\begin{tikzpicture}
    \node (E) at (-1,0) {$S$};
    \node (V) at (1,0) {$T$};
    \node (A) at (3,0) {$T^\prime$};
    \node[rotate=-90] (L) at (-0.15,0) {$\Rightarrow$};
    \node[rotate=-90] (R) at (0.15,0) {$\Rightarrow$};
    \draw[->] (E) edge [bend left=30] node [above] {$ $} (V);
    \draw[->] (E) edge [bend right=30] node [below] {$ $} (V);
    \draw[->] (V) edge node [above] {$e$} (A);
\end{tikzpicture}
$$
and, since $\Mndsf(\Cat)$ is cocomplete, $e$ is also a reflexive coequifier
$$
\begin{tikzpicture}
    \node (E) at (-1,0) {$S+T$};
    \node (V) at (1,0) {$T$};
    \node (A) at (3,0) {$T^\prime.$};
    \node[rotate=-90] (L) at (-0.15,0.2) {$\Rightarrow$};
    \node[rotate=-90] (R) at (0.15,0.2) {$\Rightarrow$};
    \draw[->] (E) edge [bend left=40] node [above] {$ $} (V);
    \draw[->] (E) edge node [below] {$ $} (V);
    \draw[->] (V) edge node [above] {$e$} (A);
    \draw[->] (V) edge [bend left=40] (E);
\end{tikzpicture}
$$
By Remark~\ref{rem:sf-monads-are-monadic}, we know that the forgetful 2-functor
$$
U: \Mndsf(\Cat) \to [N,\Cat]
$$
preserves sifted colimits, and therefore $$
e_n: Tn \twoheadrightarrow T^\prime n
$$
is a b.o.~full functor for any finite discrete category $n$.

Of course, for strongly finitary monads we may state an even stronger pointwise property of quotient monad maps: given a quotient $e: T \twoheadrightarrow T'$, its component $e_C: TC \to T'C$ is b.o.~full for every \emph{finitely presentable} category $C$. This follows since $T$ and $T'$ preserve sifted colimits, since every finitely presentable category is a sifted colimit (reflexive codescent) of finite discrete categories, and from commutativity of colimits. Moreover, every category is a filtered colimit of finitely presentable categories, and thus the component $e_C: TC \to T'C$ is b.o.~full for \emph{every} category.

In the same manner we can argue that every strongly finitary endo-2-functor $T: \Cat \to \Cat$ preserves b.o.~full functors, as they are coequifiers for some reflexive diagram in $\Cat$. We will use this fact very often in the following sections.
\end{remark}

\begin{remark}
\label{rem:alg-meaning-ptwise}
Let us give an algebraic meaning to the fact that a quotient $e: T \twoheadrightarrow T^\prime$ of strongly finitary 2-monads on $\Cat$ implies that every $e_n: Tn \twoheadrightarrow T^\prime n$ is b.o.~full in $\Cat$. Viewing the objects of $Tn$ as $n$-ary terms, bijectivity on objects of $e_n$ means that the quotient $e$ does not postulate any new equations between terms.
On the other hand, fullness of $e_n$ means that $T^\prime n$ is obtained from $Tn$ by identifying morphisms in $Tn$. On the level of algebras, this imposes equations between \emph{morphisms} of the underlying category of an algebra.
\end{remark}

\begin{example}
Let $\MonCat$ be the 2-category of monoidal categories, and $[2,0]\mbox{-}\Cat$ be the 2-category of ``monoidal categories without coherence''; that is, the objects of $[2,0]\mbox{-}\Cat$ are categories equipped with one nullary operation $I$, one binary operation $\otimes$, and two natural isomorphisms: the associator $\alpha$ with components $\alpha_{X,Y,Z}: (X \otimes Y) \otimes Z \to X \otimes (Y \otimes Z)$ and the left and right unitors $\lambda$ and $\rho$, with the components $\lambda_X: I \otimes X \to X$ and $
\rho_X: X \otimes I \to X$. Suppose that $T$ is the strongly finitary 2-monad such that $\Alg(T) \simeq [2,0]\mbox{-}\Cat$, and $T^\prime$ is the strongly finitary 2-monad such that $\Alg(\T^\prime) \simeq \MonCat$. (We know that $T$ is strongly finitary since there is a Lawvere 2-theory $\T$ for which $\Alg(T) \simeq \Alg(\T)$ holds.) There is an obvious algebraic forgetful 2-functor $U = \Alg(e): \MonCat \to [2,0]\mbox{-}\Cat$ that comes from a monad morphism $e: T \to T^\prime$, and $e$ is b.o.~full.
\end{example}

The last part of this section contains a 2-dimensional variant of the adjoint functor theorem. It will be useful when proving the Birkhoff theorem for the (b.o.~full, faithful) factorisation system in the next section.

\begin{remark}
Consider the case of ordinary categories, and suppose we are given a functor $U: \A \to \X$. For an object $X$ of $\X$, take an object $F_0 X$ of $\A$ together with a morphism $\eta_X: X \to U F_0 X$ in $\X$.

We may define $F_0 X$ to be \emph{(weakly) free on $X$} in terms of the properties of the action $b$ defined as in the diagram below
\begin{equation}
\label{eq:weak-free}
\begin{tikzcd}
\A(F_0 X, A) \ar{rr}{b} \drar[swap]{U_{F_0 X, A}} & & \X(X,UA) \\
& \X(U F_0 X, UA) \urar[swap]{\X(\eta_X,UA)}. & 
\end{tikzcd}
\end{equation}

\begin{enumerate}
\item If $b$ is an epi, then $\eta_X$ exhibits $F_0 X$ as \emph{weakly free on $X$}.
\item If $b$ is an iso, then $\eta_X$ exhibits $F_0 X$ as \emph{free on $X$}.
\end{enumerate}
\end{remark}

We shall use the above approach in the case of $\V = \Cat$ as well to define (weakly) free objects.

\begin{definition}
Let $U: \A \to \X$ be a 2-functor between 2-categories $\A$ and $\X$. We say that an object $F_0 X$ together with a morphism $\eta_X: X \to UF_0 X$ is
\begin{enumerate}
\item \emph{weakly free on $X$} if the functor $b$ in the diagram~\eqref{eq:weak-free} is a surjective on objects and full functor,
\item \emph{free on $\X$} if $b$ is an isomorphism of categories.
\end{enumerate}
\end{definition}

\begin{proposition}
\label{prop:2-dim-gaft}
Let $U$ be a
2-functor $U: \A \to \X$ between 2-categories $\A$ and $\X$. Suppose that $\A$ is 2-complete and $U$ is 2-continuous. Then $U$ has a 2-dimensional left adjoint $F: \X \to \A$ if for every object $X$ of $\X$ there exists a \emph{2-solution set}
$$
S_X = \{ X \xrightarrow{f_i} UA_i \mid i \in I \},
$$
meaning that the functor
\begin{align*}
s: \Coprod_{i \in I} \A(A_i,A) &\to \X(X,UA) \\
h & \mapsto Uh \cdot f_i \\
\alpha: h \Rightarrow h' &\mapsto U \alpha * f_i: Uh \cdot f_i \Rightarrow Uh' \cdot f_i
\end{align*}
is surjective on objects and full.
\end{proposition}

Before proceeding with the proof, let us describe the requirements for a 2-solution set in elementary terms.
\begin{enumerate}
\item Surjectivity on objects of $s$. This requirement corresponds to the 1-dimensional solution set condition. That is, for every morphism $f: X \to UA$ there exists a morphism $f_i: X \to UA_i$ from $S_X$ and a morphism $h: A_i \to A$ such that $f = Uh \cdot f_i$ holds. In a diagram:
$$
\begin{tikzcd}
X \rar{\exists f_i} \drar[swap]{\forall f}  & UA_i \dar{Uh} & A_i \dar{\exists h} \\
  & UA & A
\end{tikzcd}
$$
\item Fullness of $s$. This requirement corresponds to a 2-cell factorisation property. That is, for every 2-cell
$$
\begin{tikzpicture}
    \node (X) at (-1,0) {$X$};
    \node (V) at (1,0) {$UA$};
    \node[rotate=-90] (L) at (0,0) {$\Rightarrow$};
    \node (alpha) at (0.3,0) {$\phi$};
    \draw[->] (X) edge [bend left=30] node [above] {$f$} (V);
    \draw[->] (X) edge [bend right=30] node [below] {$f'$} (V);
\end{tikzpicture}
$$
there exists a 2-cell
$$
\begin{tikzpicture}
    \node (X) at (-1,0) {$A_i$};
    \node (V) at (1,0) {$A$};
    \node[rotate=-90] (L) at (0,0) {$\Rightarrow$};
    \node (alpha) at (0.3,0) {$\psi$};
    \draw[->] (X) edge [bend left=30] node [above] {$h$} (V);
    \draw[->] (X) edge [bend right=30] node [below] {$h'$} (V);
\end{tikzpicture}
$$
for some $i \in I$ such that $\phi = U \psi * f_i$.
\end{enumerate}

\begin{proof}[Proof of Proposition~\ref{prop:2-dim-gaft}]
The 1-dimensional property comes from the ordinary adjoint functor theorem, as it is proved, e.g., in~\cite{manes}. We denote by $A_0$ the product $\prod_{i \in I} A_i$ with projections $\pi_j: \prod_{i \in I} A_i \to A_j$. Since $U$ is 2-continuous, we have an isomorphism
$$U A_0 = U \prod_{i \in I} A_i \cong \prod_{i \in I} U A_i.$$
Let us define a morphism
$$
\langle f_i \rangle : X \to \prod_{i \in I} U A_i
$$
by the universal property of the product by demanding the equations $U \pi_j \cdot \langle f_i \rangle = f_j$ to hold. Denote by
\begin{center}
\begin{tikzpicture}
    \node (X) at (0,0) {$A_0$};
    \node (Y) at (2,0) {$A_0$};
    \node (ell) at (1,0) {\rotatebox{90}{$\cdots$}};
    
    \draw[transform canvas={yshift=1.6ex},->] (X) to node [above] {$h$} (Y);
    \draw[transform canvas={yshift=-1.6ex},->] (X) to node [below] {$h'$} (Y);
\end{tikzpicture}
\end{center}
the set of morphisms for which the diagram
\begin{center}
\begin{tikzpicture}
    \node (X) at (0,0) {$U A_0$};
    \node (Y) at (2,0) {$U A_0$};
    \node (ell) at (1,0) {\rotatebox{90}{$\cdots$}};
    
    \node (realX) at (1,-2) {$X$};
    \draw[->] (realX) to node [left] {$\langle f_i \rangle$} (X);
    \draw (realX) to node [right] {$\langle f_i \rangle$} (Y);
    
    \draw[transform canvas={yshift=1.6ex},->] (X) to node [above] {$Uh$} (Y);
    \draw[transform canvas={yshift=-1.6ex},->] (X) to node [below] {$Uh'$} (Y);
\end{tikzpicture}
\end{center}
commutes. We denote by
\begin{center}
\begin{tikzpicture}
    \node (X) at (0,0) {$A_0$};
    \node (Y) at (2,0) {$A_0$};
    \node (ell) at (1,0) {\rotatebox{90}{$\cdots$}};
    
    \node (F0X) at (-2,0) {$F_0 X$};
    \draw[->] (F0X) to node [above] {$e$} (X);
    
    \draw[transform canvas={yshift=1.6ex},->] (X) to node [above] {$h$} (Y);
    \draw[transform canvas={yshift=-1.6ex},->] (X) to node [below] {$h'$} (Y);
\end{tikzpicture}
\end{center}
the collective equaliser of the morphisms $h,\dots,h'$. Since $U$ is 2-continuous, the morphism $Ue$ is the collective equaliser of $Uh, \dots, Uh'$. By definition, the morphism ${\langle f_i \rangle}$ equalises $Uh, \dots, Uh'$, and thus we can define a morphism
$$
\eta_X: X \to UF_0 X.
$$
by the universal property of $Ue$. Let us first show that $X \xrightarrow{\langle f_i \rangle} UA_0$ is (2-dimensionally) weakly free on $X$. Given a situation
$$
\begin{tikzpicture}
    \node (X) at (-1,0) {$X$};
    \node (V) at (1,0) {$UA$};
    \node[rotate=-90] (L) at (0,0) {$\Rightarrow$};
    \node (alpha) at (0.3,0) {$\phi$};
    \draw[->] (X) edge [bend left=30] node [above] {$f$} (V);
    \draw[->] (X) edge [bend right=30] node [below] {$f'$} (V);
\end{tikzpicture}
$$
there exists a 2-cell
$$
\begin{tikzpicture}
    \node (X) at (-1,0) {$A_0$};
    \node (V) at (1,0) {$A$};
    \node[rotate=-90] (L) at (0,0) {$\Rightarrow$};
    \node (alpha) at (0.3,0) {$\psi$};
    \draw[->] (X) edge [bend left=30] node [above] {$h$} (V);
    \draw[->] (X) edge [bend right=30] node [below] {$h'$} (V);
\end{tikzpicture}
$$
such that
$$
U \psi \cdot \langle f_i \rangle = \phi.
$$
Indeed, by the 2-dimensional solution set condition we just use a projection from $A_0 = \prod A_i$.

Secondly, we will show that $X \xrightarrow{\eta_X} UF_0 X$ is weakly 2-dimensionally free. This is true since $\langle f_i \rangle$ is weakly 2-dimensionally free and since $\langle f_i \rangle$ factorises as $Ue \cdot \eta_X$. Thus for any morphism $f: X \to UA$ we have a morphism $h: A_0 \to A$ such that the right hand side triangle in the following diagram
$$
\begin{tikzcd}
UF_0 X \rar{Ue} & UA_0 \rar{Uh} & UA \\
& X \ular{\eta_X} \uar{\langle f_i \rangle} \arrow[swap]{ur}{f} &
\end{tikzcd}
$$
commutes, and $h \cdot e: F_0 X \to A$ witnesses the 1-dimensional property of weak freeness for $F_0 X$. Checking the 2-dimensional property is analogous. 

Thirdly, we show that $X \xrightarrow{\eta_X} UF_0 X$ is 2-dimensionally free. We know that $X \xrightarrow{\eta_X} UF_0 X$ is ordinarily free by the ordinary general adjoint functor theorem. For the 2-dimensional aspect, suppose the equalities
$$
\begin{tikzpicture}
    \node (X) at (0,1) {$X$};
    \node (V) at (0,-1) {$UA$};
    \node (L) at (0,0) {$\Rightarrow$};
    \node (alpha) at (0,0.3) {$\phi$};
    \draw[->] (X) edge [bend left=30] node [right] {$f'$} (V);
    \draw[->] (X) edge [bend right=30] node [left] {$f\phantom{'}$} (V);

\node (eq) at (3,0) {$=$};

\begin{scope}[shift={(6,0)}]
    \node (X) at (0,1) {$U F_0 X$};
    \node (V) at (0,-1) {$UA$};
    \node (Y) at (-2,1) {$X$};
    \node (L) at (0,0) {$\Rightarrow$};
    \node (alpha) at (0,0.3) {$\psi_i$};
    \draw[->] (X) edge [bend left=30] node [right] {$U (f'^\sharp)$} (V);
    \draw[->] (X) edge [bend right=30] node [left] {$U (f^\sharp)$} (V);
    \draw[->] (Y) edge node [above] {$\eta_X$} (X);
\end{scope}
\end{tikzpicture}
$$
hold for $\psi_i$ with $i \in \{1,2\}$.
We then need to prove that $\psi_1 = \psi_2$ holds. For this it is enough to show that in the equifier
$$
\begin{tikzpicture}
    \node (Z) at (-3,0) {$E$};
    \node (E) at (-1,0) {$F_0 X$};
    \node (V) at (1,0) {$A$};
    \node[rotate=-90] (L) at (-0.15,0) {$\Rightarrow$};
    \node[rotate=-90] (R) at (0.15,0) {$\Rightarrow$};
    \node (alpha) at (-0.4,0) {$\psi_1$};
    \node (beta) at (0.4,0) {$\psi_2$};
    \draw[->] (Z) edge node [above] {$j$} (E);
    \draw[->] (E) edge [bend left=30] node [above] {$f^\hash$} (V);
    \draw[->] (E) edge [bend right=30] node [below] {${f'}^\hash$} (V);
\end{tikzpicture}
$$
the morphism $j$ is iso. By 2-continuity of $U$, we have that
$$
\begin{tikzpicture}
    \node (Z) at (-3,0) {$UE$};
    \node (E) at (-1,0) {$UF_0 X$};
    \node (V) at (2,0) {$UA$};
    \node[rotate=-90] (L) at (0.45,0) {$\Rightarrow$};
    \node[rotate=-90] (R) at (0.65,0) {$\Rightarrow$};
    \node (alpha) at (0,0) {$U\psi_1$};
    \node (beta) at (1.1,0) {$U\psi_2$};
    \draw[->] (Z) edge node [above] {$Uj$} (E);
    \draw[->] (E) edge [bend left=30] node [above] {$Uf^\hash$} (V);
    \draw[->] (E) edge [bend right=30] node [below] {$U{f'}^\hash$} (V);
\end{tikzpicture}
$$
is also an equifier. Moreover, $\eta_X$ equifies $U \psi_1$ and $U \psi_2$. Thus we have a triangle 
$$
\begin{tikzpicture}
    \node (Z) at (-3,0) {$UE$};
    \node (E) at (-1,0) {$UF_0 X$};
    \node (V) at (2,0) {$UA$};
    \node[rotate=-90] (L) at (0.45,0) {$\Rightarrow$};
    \node[rotate=-90] (R) at (0.65,0) {$\Rightarrow$};
    \node (alpha) at (0,0) {$U\psi_1$};
    \node (beta) at (1.1,0) {$U\psi_2$};
    
    \node (X) at (-2,-1.5) {$X$};    
    \draw[->, dotted] (X) edge node [below left] {$h$} (Z);
    \draw[->] (X) edge node [below right] {$\eta_X$} (E);
    
    \draw[->] (Z) edge node [above] {$Uj$} (E);
    \draw[->] (E) edge [bend left=30] node [above] {$Uf^\hash$} (V);
    \draw[->] (E) edge [bend right=30] node [below] {$U{f'}^\hash$} (V);
\end{tikzpicture}
$$
from the universal property of $Uj$. Now we see that the diagram
$$
\begin{tikzcd}
UA_0 \rar{Uk} & UE \rar{Uj} & UF_0 X \dar{Ue} \\
& X \ular{\langle f_i \rangle} \uar{h} \urar[swap]{\eta_X} \rar[swap]{\langle f_i \rangle} & UA_0
\end{tikzcd}
$$
commutes, as the triangles commute (from left to right) by ordinary weak freeness of $\langle f_i \rangle:  X \to UA_0$, by the definition of $h$ and by the definition of $\eta_X$.

Thus both the morphisms $\id_{A_0}: A_0 \to A_0$ and
$$
A_0 \xrightarrow{k} E \xrightarrow{j} F_0 X \xrightarrow{e} A_0
$$
are present in the (collective) equaliser diagram defining the morphism $e$, and
$$
\begin{tikzcd}
F_0 X \rar{e} & A_0 \ar[bend left]{rrr}{\id_{A_0}} \rar{k} & E \rar{j} & F_0 X \rar{e} & A_0
\end{tikzcd}
$$
commutes. By the ordinary freeness of $F_0 X$ the diagram
$$
\begin{tikzcd}
F_0 X \rar{e} \ar[bend left]{rrr}{\id_{F_0 X}} & A_0 \rar{k} & E \rar{j} & F_0 X
\end{tikzcd}
$$
commutes as well. Since $e$ is mono, it follows that $j$ is split epi. But $j$ is mono, being an equifier. Therefore $j$ is an isomorphism, as we wanted to show.
\end{proof}

\begin{remark}
\label{rem:2-dim-aft-usage}
We will use the above proposition in the following way. A 2-functor $U: \A \to \X$ has a 2-dimensional adjoint if and only if
\begin{enumerate}
\item $U$ has an ordinary adjoint $F$ witnessed by the natural isomorphism
$$
({-})^\sharp: \X(X,UA) \to \A(FX,A)
$$
of sets, and
\item for every $X$ in $\X$ and a 2-cell
$$
\begin{tikzpicture}
    \node (E) at (-1,0) {$X$};
    \node (V) at (1,0) {$UA,$};
    \node[rotate=-90] (L) at (0,0) {$\Rightarrow$};
    \node (alpha) at (0.3,0) {$\alpha$};
    \draw[->] (E) edge [bend left=30] node [above] {$f$} (V);
    \draw[->] (E) edge [bend right=30] node [below] {$f'$} (V);
\end{tikzpicture}
$$
there is a 2-cell $\alpha^\sharp: f^\sharp \Rightarrow (f')^\sharp$ with $\alpha = U \alpha^\sharp * \eta_X$.
\end{enumerate}
This criterion is useful in the situations where the ordinary adjoint is easy to obtain.
\end{remark}

\section{Birkhoff theorem for the kernel-quotient system $\F_\bof$}

In this section we state and prove the Birkhoff theorem for the $\F_\bof$ kernel-quotient system. We will first recall basic definitions concerning subcategories and equivalence of categories in the $\Cat$-enriched setting. Then we review the 2-dimensional aspects of algebraic categories and algebraic functors, some important properties of algebraic categories, and we define the right notion of a subalgebra and a quotient algebra that is used in the main theorem.

\begin{definition}
A 2-functor $T: \K \to \LL$ between 2-categories $\K$ and $\LL$ is called \emph{fully faithful} if for any pair $X,Y$ of objects of $\X$ the action $T_{X,Y}: \K(X,Y) \to \LL(TX,TY)$ is an \emph{isomorphism} of categories. We say that $T$ exhibits $\K$ as a \emph{full subcategory} of $\LL$. When $\K$ is moreover closed under isomorphisms in $\LL$, we call $\K$ a \emph{replete} full subcategory of $\LL$. By closure under isomorphisms we mean that for any object $X$ in $\K$ and any isomorphism $i: TX \to Z$ in $\LL$ there exists an isomorphism $j: X \to Y$ in $\K$ such that $Tj = i$.

The 2-categories $\K$ and $\LL$ are \emph{equivalent} if there is a fully faithful 2-functor $T: \K \to \LL$ that is \emph{essentially surjective}, that is, for any object $Z$ of $\LL$ there exists an object $X$ of $\K$ with $TX \cong Z$.
\end{definition}

\begin{remark}
For a 2-monad $T$ on a 2-category $\C$, we denote the 2-category of $T$-algebras and their homomorphisms by $\Alg(T)$. We call the categories equivalent to the categories of the form $\Alg(T)$ \emph{algebraic}. Let us recall the 2-dimensional structure of $\Alg(T)$. Given two $T$-algebras $a: TA \to A$ and $b: TB \to B$, and two homomorphisms $h,h^\prime: A \to B$ between $(A,a)$ and $(B,b)$, the 2-cells $\alpha: h^\prime \Rightarrow h$ between the homomorphisms $h^\prime$ and $h$ are exactly those 2-cells $\alpha: h^\prime \Rightarrow h$ between the morphisms $h^\prime$ and $h$ in $\Cat$ that moreover satisfy the following equality:
$$
\begin{tikzpicture}
    \node (X) at (0,1) {$TA$};
    \node (V) at (2,1) {$TB$};
    \node (Y) at (2,-1) {$B$};
    \node[rotate=-90] (L) at (0.9,1) {$\Rightarrow$};
    \node (alpha) at (1.3,1) {$T\alpha$};
    \draw[->] (X) edge [bend left=30] node [above] {$Th'$} (V);
    \draw[->] (X) edge [bend right=30] node [below] {$Th\phantom{'}$} (V);
    \draw[->] (V) edge node [right] {$b$} (Y);

\node (eq) at (3,0) {$=$};

\begin{scope}[shift={(6,0)}]
    \node (X) at (-2,-1) {$A$};
    \node (V) at (0,-1) {$B$};
    \node (Y) at (-2,1) {$TA$};
    \node[rotate=-90] (L) at (-1,-1) {$\Rightarrow$};
    \node (alpha) at (-0.7,-1) {$\alpha$};
    \draw[->] (X) edge [bend left=30] node [above] {$h'$} (V);
    \draw[->] (X) edge [bend right=30] node [below] {$h\phantom{'}$} (V);
    \draw[->] (Y) edge node [left] {$a$} (X);
\end{scope}
\end{tikzpicture}
$$
\end{remark}

\begin{remark}
\label{rem:alg}
Consider two 2-monads $T$ and $T'$ on $\Cat$, and a monad morphism $e: T \to T'$. This monad morphism gives rise to an algebraic 2-functor $\Alg(e): \Alg(T') \to \Alg(T)$ between the 2-categories $\Alg(T')$ and $\Alg(T)$ of algebras for $T'$ and $T$. On objects, $\Alg(e)$ acts as follows:
$$
\begin{tikzpicture}
    \node (TpA) at (0,1) {$T'A$};
    \node (A) at (0,-1) {$A$};
    \node (TA) at (2,1) {$TA$};
    \node (TpA2) at (2,0) {$T'A$};
    \node (A2) at (2,-1) {$A$};
    
\node (mapsto) at (1,0) {$\mapsto$};

	\draw[->] (TpA) edge node [left] {$a'$} (A);
	\draw[->] (TA) edge node [right] {$e_A$} (TpA2);
	\draw[->] (TpA2) edge node [right] {$a'$} (A2);
\end{tikzpicture}
$$
On morphisms and 2-cells $\Alg(e)$ acts as an identity. A homomorphism $h: (A,a') \to (B,b')$ between two $T^\prime $-algebras $a': T^\prime A \to A$ and $b': T^\prime B \to B$ gets mapped to a homomorphism $h: (A, a' \cdot e_A) \to (B,b' \cdot e_B)$ of the corresponding $T$-algebras. Indeed, the outer rectangle in the diagram
$$
\xymatrix{
TA \ar@{->}[d]_{e_A} \ar[r]^{Th} & TB \ar@{->}[d]^{e_B}\\
T^\prime A \ar[d]_{a'} \ar[r]^{T^\prime h} & T^\prime B \ar[d]^{b'}\\
A \ar[r]_h & B
}
$$
clearly commutes. The same reasoning applies for the 2-cells $\alpha: h \to h^\prime$ between two homomorphisms $h: (A,a') \to (B,b')$ and $h^\prime: (A,a') \to (B,b')$.

The action of $\Alg(e)$ on morphisms and 2-cells is thus faithful.
\end{remark}

We will state Birkhoff's theorem using certain closure properties, including the closure under \emph{quotient algebras} and \emph{subalgebras}.

\begin{definition}
Let us fix a strongly finitary 2-monad $T$ from $\Mndsf(\Cat)$, and take an algebra $a: TA \to A$ from $\Alg(T)$. We say that $b: TB \to B$ is a \emph{subalgebra} of $a: TA \to A$ if there is a faithful functor $m: A \rightarrowtail B$ such that the diagram
$$
\xymatrix{
T B \ar[d]_b \ar[r]^{T m} & T A \ar[d]^{a}\\
B \ar@{ >->}[r]_m & A
}
$$
commutes.
By a \emph{quotient algebra} of an algebra $a: TA \to A$ we mean an algebra $b: TB \to B$ together with a b.o.~full morphism $h: A \twoheadrightarrow B$ in $\Cat$ that is a homomorphism, i.e., the diagram
$$
\xymatrix{
T A \ar[d]_a \ar@{->>}[r]^{T h} & T B \ar[d]^{b}\\
A \ar@{->>}[r]_h & B
}
$$
commutes.
\end{definition}

Let us remark that in the above diagram concerning quotient algebras, the morphism $Th: TA \twoheadrightarrow TB$ is indeed b.o.~full by Remark~\ref{rem:sf-preserve-bof} since $h$ is and $T$ is strongly finitary.

\begin{remark}
\label{rem:cowelpoweredness}
The 2-category $\Alg(T)$ of algebras for a strongly finitary 2-monad $T$ is cowellpowered with respect to quotient algebras. Indeed, for every small category $A$ there is, up to isomorphism, only a set of b.o.~full functors of the form $e: A \to B$ in $\Cat$. Thus for a $T$-algebra $(A,a)$ there is, up to isomorphism, only a set of quotients $e: (A,a) \to (B,b)$ in $\Alg(T)$.
\end{remark}

\begin{remark}
\label{rem:creation}
Given an algebraic 2-category $\Alg(T)$ for a strongly finitary 2-monad $T$ on $\Cat$, it is a standard observation that the underlying 2-functor $U: \Alg(T) \to \Cat$ creates limits. Since $T$ is strongly finitary, the 2-functor $U$ also creates sifted colimits. In particular, $U$ creates reflexive coequifiers. That is, given a reflexive diagram
$$
\begin{tikzpicture}
    \node (E) at (-1,0) {$(K,k)$};
    \node (V) at (1,0) {$(A,a)$};
    \node[rotate=-90] (L) at (-0.15,0.2) {$\Rightarrow$};
    \node[rotate=-90] (R) at (0.15,0.2) {$\Rightarrow$};
    \draw[->] (E) edge [bend left=40]
    node [above] {$f$}
    (V);
    \draw[->] (E) edge
    node [below] {$g$}
    (V);
    \draw[->] (V) edge [bend left=40]
    node [below] {$h$}
    (E);
\end{tikzpicture}
$$
in $\Alg(T)$, and the coequifier of the $U$-image of the above diagram
$$
\begin{tikzpicture}
    \node (E) at (-1,0) {$K$};
    \node (V) at (1,0) {$A$};
    \node (A) at (3,0) {$C$};
    \node[rotate=-90] (L) at (-0.15,0.2) {$\Rightarrow$};
    \node[rotate=-90] (R) at (0.15,0.2) {$\Rightarrow$};
    \draw[->] (E) edge [bend left=40]
    node [above] {$Uf$}
    (V);
    \draw[->] (E) edge
    node [below] {$Ug$}
    (V);
    \draw[->] (V) edge node [above] {$c$} (A);
    \draw[->] (V) edge [bend left=40]
    node [below] {$Uh$}
    (E);
\end{tikzpicture}
$$
there exists a unique algebra $(C,c)$ such that $c$ is a homomorphism between $(A,a)$ and $(C,c)$.
\end{remark}

We now turn to the proof of the Birkhoff theorem. 

\begin{theorem}[Two-dimensional Birkhoff theorem]
\label{thm:birkhoff}
Let $T$ be a strongly finitary 2-monad on $\Cat$ and let $\A$ be a full subcategory $\Alg(T)$ of the category of algebras for the 2-monad $T$. Then the following are equivalent:
\begin{enumerate}
\item There is a strongly finitary 2-monad $T^\prime $ and a b.o.~full monad morphism $e: T \twoheadrightarrow T^\prime $ such that the comparison 2-functor $\A \to \Alg(T^\prime)$ is an equivalence.
\item The category $\A$ is closed in $\Alg(T)$ under sifted colimits, 2-products, quotient algebras, and subalgebras.
\end{enumerate}
\end{theorem}

\begin{proof}
We first prove the implication (1) $\Rightarrow$ (2) in the following manner:
\begin{enumerate}[(a)]
\item Given the monad morphism $e: T \twoheadrightarrow T^\prime $, we get a 2-functor $\Alg(e): \Alg(T^\prime ) \to \Alg(T)$ that we show to be fully faithful.
\item We show that $\Alg(e)$ preserves sifted colimits and limits.
\item Finally we show that $\Alg(T^\prime )$ is closed in $\Alg(T)$ under subalgebras and quotient algebras.
\end{enumerate}

Ad (a): The action of $\Alg(e)$ on morphisms and 2-cells is faithful by Remark~\ref{rem:alg}. We prove that $\Alg(e)$ is indeed fully faithful by showing that its action on morphisms and 2-cells is full. The fullness on morphisms comes from observing that given any diagram of the form
$$
\xymatrix{
TA \ar@{->>}[d]_{e_A} \ar[r]^{Th} & TB \ar@{->>}[d]^{e_B}\\
T^\prime A \ar[d]_{a'} \ar[r]^{T^\prime h} & T^\prime B \ar[d]^{b'}\\
A \ar[r]_h & B
}
$$
such that the outer rectangle commutes, we can show that the square
$$
\xymatrix{
T^\prime A \ar[d]_{a'} \ar[r]^{T^\prime h} & T^\prime B \ar[d]^{b'}\\
A \ar[r]_h & B
}
$$
commutes as well. The morphism $e_A: TA \twoheadrightarrow T^\prime A$ is b.o.\ full, and thus epi. The above square therefore commutes by the cancellation property of $e_A$. Similarly, given a 2-cell $\alpha: h \Rightarrow h^\prime$ satisfying the equality
$$
(b' \cdot e_B) * T \alpha = \alpha * (a' \cdot e_A),
$$
we infer that
$$
b' * T^\prime \alpha = \alpha * a'
$$
holds by 2-naturality of $e$ (saying that $e_B * T \alpha = \alpha * e_a$), and from $e_A$ being an epimorphism on 2-cells by~Lemma~\ref{lem:cancel-2-cells}. The algebraic 2-functor $\Alg(e)$ is therefore indeed fully faithful.

Ad (b): Since $\Alg(T^\prime)$ is 2-complete and 2-cocomplete, the algebraic 2-functor $\Alg(e)$ has a left adjoint by Theorem~3.9 of~\cite{bkp}, and thus preserves limits. It also preserves sifted colimits  since these are computed on the level of underlying categories in both $\Alg(T^\prime)$ and $\Alg(T)$.

Ad (c): Now we show that the 2-category $\Alg(T^\prime)$ is closed in $\Alg(T)$ under subalgebras and quotient algebras. To this end, consider an algebra $a': T^\prime A \to A$ and its image $a= a' \cdot e_A: TA \to A$ under $\Alg(e)$. Given any subalgebra $b: TB \to B$
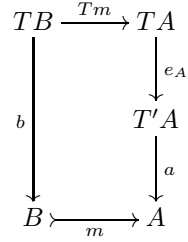
$$
\xymatrix{
TB \ar[dd]_{b} \ar[r]^{Tm} & TA \ar[d]^{e_A} \\
& T^\prime A \ar[d]^{a} \\
B \ar@{ >->}[r]_{m} & A
}
$$
of $(A, a)$, we can use the naturality of $e$
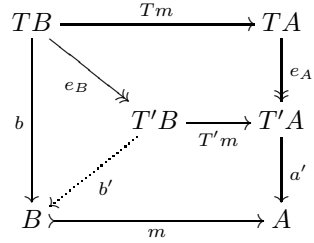
$$
\xymatrix{
TB \ar@{->>}[dr]_{e_B} \ar[dd]_{b} \ar[rr]^{Tm} & & TA \ar@{->>}[d]^{e_A} \\
& T^\prime B \ar@{.>}[dl]^{b^\prime} \ar[r]_{T^\prime m} & T^\prime A \ar[d]^{a'} \\
B \ar@{ >->}[rr]_{m} & & A
}
$$
and define $b^\prime$ as the unique diagonal fill-in with respect to $e_B$ and $m$ in the above diagram. (Recall that (b.o.~full, faithful) is a factorisation system on $\Cat$.) This $b^\prime: T^\prime B \to B$ is a $T^\prime$-algebra: it satisfies both the algebra axioms. Firstly, the lower triangle in the following diagram
$$
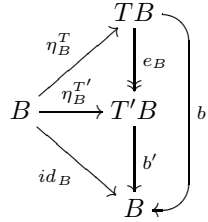
\xymatrix{
& TB \ar@{->>}[d]^{e_B} \ar@{->} `r[d] `[dd]^{b} [dd] \\
B \ar[ur]^{\eta^T_B} \ar[r]^{\eta^{T^\prime}_B} \ar[dr]_{\id_B} & T^\prime B \ar[d]^{b^\prime} \\
& B
}
$$
commutes since the upper one commutes by the unit axiom of a monad morphism and the outer one commutes as $(B,b)$ is a $T$-algebra. The outer square in the following diagram
$$
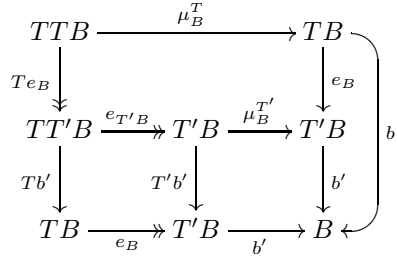
\xymatrix{
TTB \ar[rr]^{\mu^T_B} \ar@{->>}[d]_{Te_B} & & TB \ar[d]^{e_B} \ar@{->} `r[d] `[dd]^{b} [dd] \\
TT^\prime B \ar@{->>}[r]^{e_{T^\prime B}} \ar[d]_{Tb^\prime} & T^\prime B \ar[d]_{T^\prime b^\prime} \ar[r]^{\mu^{T^\prime}_B} & T^\prime B \ar[d]^{b^\prime} \\
TB \ar@{->>}[r]_{e_B} & T^\prime B \ar[r]_{b^\prime} & B
}
$$
commutes since $(B,b)$ is a $T$-algebra. The upper rectangle is an instance of a monad morphism axiom, and the lower left square commutes by naturality of $e$. The morphism $Te_B$ is b.o.\ full, as $e_B$ is and $T$ preserves them by Remark~\ref{rem:sf-preserve-bof}. Thus the composite morphism $e_{T^\prime B} \cdot Te_B$ is b.o.~full as well. By the cancellation property of b.o.~full morphisms we obtain the commutativity of the square
$$
\xymatrix{
T^\prime B \ar[d]_{T^\prime b^\prime} \ar[r]^{\mu^{T^\prime}_B} & T^\prime B \ar[d]^{b^\prime} \\
T^\prime B \ar[r]_{b^\prime} & B,
}
$$
and this proves that $(B,b')$ is a $T^\prime$-algebra. In conclusion, $\Alg(T^\prime )$ is indeed closed in $\Alg(T)$ under subalgebras.

The closure of $\Alg(T^\prime)$ under quotient algebras in $\Alg(T)$ follows from closure under limits and sifted colimits. We are given a $T^\prime$-algebra $a': T^\prime A \to A$ and a quotient homomorphism $h: (A,a ) = (A, a' \cdot e_A) \twoheadrightarrow (B,b)$ of $T$-algebras as in the diagram
$$
\xymatrix{
TA \ar@{->>}[d]_{e_A} \ar@{->>}[r]^{Th} & TB \ar@{->}[dd]^{b}\\
T^\prime A \ar[d]_{a'} & \\
A \ar@{->>}[r]_h & B
}
$$
We can form a kernel
$$
\begin{tikzpicture}
    \node (E) at (-1,0) {$(K,k)$};
    \node (V) at (1,0) {$(A,a)$};
    \node[rotate=-90] (L) at (-0.15,0.2) {$\Rightarrow$};
    \node[rotate=-90] (R) at (0.15,0.2) {$\Rightarrow$};
    \draw[->] (E) edge [bend left=40] node [above] {$f$} (V);
    \draw[->] (E) edge node [below] {$g$} (V);
    \draw[->] (V) edge [bend left=40] (E);
\end{tikzpicture}
$$
of $h$ in $\Alg(T)$ that is reflexive by construction, and whose coequifier is $h$ itself. We therefore have a diagram
$$
\begin{tikzpicture}
    \node (TK) at (-1,4) {$TK$};
    \node (TA) at (1,4) {$TA$};
    \node (TB) at (3,4) {$TB$};

	\node (TpA) at (1,2) {$T^\prime A$};

    \node (K) at (-1,0) {$K$};
    \node (A) at (1,0) {$A$};
    \node (B) at (3,0) {$B$};

    \node[rotate=-90] (L) at (-0.15,0.2) {$\Rightarrow$};
    \node[rotate=-90] (R) at (0.15,0.2) {$\Rightarrow$};
    \node[rotate=-90] (L) at (-0.15,4.2) {$\Rightarrow$};
    \node[rotate=-90] (R) at (0.15,4.2) {$\Rightarrow$};

    \draw[->] (TK) edge [bend left=40] node [above] {$Tf$} (TA);
    \draw[->] (TK) edge node [below] {$Tg$} (TA);
    \draw[->] (TA) edge [bend left=40] (TK);
    \draw[->] (TA) edge node [above] {$Th$} (TB);

	\draw[->] (TK) edge node [left] {$k$} (K);
	\draw[->] (TA) edge [bend left=30] node [right] {$a$} (A);
	\draw[->>] (TA) edge node [left] {$e_A$} (TpA);
	\draw[->] (TpA) edge node [left] {$a'$} (A);
	\draw[->] (TB) edge node [right] {$b$} (B);

    \draw[->] (K) edge [bend left=40] node [above] {$f$} (A);
    \draw[->] (K) edge node [below] {$g$} (A);
    \draw[->] (A) edge [bend left=40] (K);
    \draw[->] (A) edge node [below] {$h$} (B);
\end{tikzpicture}
$$
where the diagram
$$
\begin{tikzpicture}
    \node (E) at (-1,0) {$K$};
    \node (V) at (1,0) {$A$};
    \node[rotate=-90] (L) at (-0.15,0.2) {$\Rightarrow$};
    \node[rotate=-90] (R) at (0.15,0.2) {$\Rightarrow$};
    \draw[->] (E) edge [bend left=40] node [above] {$f$} (V);
    \draw[->] (E) edge node [below] {$g$} (V);
    \draw[->] (V) edge [bend left=40] (E);
\end{tikzpicture}
$$
is the kernel of $h$ in $\Cat$. Now we can use naturality of $e$ and observe that
$$
\begin{tikzpicture}
    \node (TK) at (-1,4) {$TK$};
    \node (TA) at (1,4) {$TA$};
    \node (TB) at (3,4) {$TB$};

    \node (TpK) at (-1,2) {$T^\prime K$};
    \node (TpA) at (1,2) {$T^\prime A$};
    \node (TpB) at (3,2) {$T^\prime B$};

    \node (K) at (-1,0) {$K$};
    \node (A) at (1,0) {$A$};
    \node (B) at (3,0) {$B$};

    \node[rotate=-90] (L) at (-0.15,0.2) {$\Rightarrow$};
    \node[rotate=-90] (R) at (0.15,0.2) {$\Rightarrow$};
    \node[rotate=-90] (L) at (-0.15,2.2) {$\Rightarrow$};
    \node[rotate=-90] (R) at (0.15,2.2) {$\Rightarrow$};
    \node[rotate=-90] (L) at (-0.15,4.2) {$\Rightarrow$};
    \node[rotate=-90] (R) at (0.15,4.2) {$\Rightarrow$};

    \draw[->] (TK) edge [bend left=40] node [above] {$Tf$} (TA);
    \draw[->] (TK) edge node [below] {$Tg$} (TA);
    \draw[->] (TA) edge [bend left=40] (TK);
    \draw[->] (TA) edge node [above] {$Th$} (TB);

	\draw[->] (TpK) edge [bend left=40] node [above] {$T^\prime f$} (TpA);
    \draw[->] (TpK) edge node [below] {$T^\prime g$} (TpA);
    \draw[->] (TpA) edge [bend left=40] (TpK);
    \draw[->] (TpA) edge node [below] {$T^\prime h$} (TpB);

	\draw[->] (TK) edge [bend right=30] node [left] {$k$} (K);
	\draw[->>] (TK) edge node [right] {$e_K$} (TpK);
	\draw[->] (TpK) edge [dotted] node [right] {$\exists! k'$} (K);
	\draw[->>] (TA) edge node [left] {$e_A$} (TpA);
	\draw[->] (TpA) edge node [left] {$a'$} (A);
	\draw[->] (TB) edge [bend left=30] node [right] {$b$} (B);
	\draw[->>] (TB) edge node [left] {$e_B$} (TpB);
	\draw[->] (TpB) edge [dotted] node [left] {$\exists! b'$} (B);

    \draw[->] (K) edge [bend left=40] node [above] {$f$} (A);
    \draw[->] (K) edge node [below] {$g$} (A);
    \draw[->] (A) edge [bend left=40] (K);
    \draw[->] (A) edge node [below] {$h$} (B);
\end{tikzpicture}
$$
there is a morphism $k'$ defined by the universal property of the kernel $K$, and by uniqueness, $k = k' \cdot e_K$ holds. Thus $(K,k')$ is a $T^\prime$-algebra, since $(K,k)$ is a $T$-algebra.
The existence of $k'$ in turn induces a morphism $b': T^\prime B \to B$ by the universal property of the coequifier $T^\prime h$. Moreover, $(B,b')$ is a $T^\prime$-algebra, as it is the coequifier of
$$
\begin{tikzpicture}
    \node (E) at (-1,2) {$T^\prime K$};
    \node (V) at (1,2) {$T^\prime A$};
    \node[rotate=-90] (L) at (-0.15,2) {$\Rightarrow$};
    \node[rotate=-90] (R) at (0.15,2) {$\Rightarrow$};
    \draw[->] (E) edge [bend left=30] node [above] {$T^\prime f$} (V);
    \draw[->] (E) edge [bend right=30] node [below] {$T^\prime g$} (V);

    \node (E1) at (-1,0) {$K$};
    \node (V1) at (1,0) {$A$};
    \node[rotate=-90] (L1) at (-0.15,0) {$\Rightarrow$};
    \node[rotate=-90] (R1) at (0.15,0) {$\Rightarrow$};
    \draw[->] (E1) edge [bend left=30] node [above] {$f$} (V1);
    \draw[->] (E1) edge [bend right=30] node [below] {$g$} (V1);
    
    \draw[->] (E) edge node [left] {$k'$} (E1);
    \draw[->] (V) edge node [right] {$a'$} (V1);
\end{tikzpicture}
$$
by creation of coequifiers of reflexive pairs (see Remark~\ref{rem:creation}). So $\Alg(T^\prime)$ is closed in $\Alg(T)$ under quotient algebras.

\medskip

The second part of the proof is the implication (2) $\Rightarrow$ (1). Given a strongly finitary 2-monad $T$ and a full replete subcategory
$$
J: \A \to \Alg(T)
$$
of $\Alg(T)$ that is closed under sifted colimits, 2-products, quotient algebras and subalgebras, we need to find a strongly finitary 2-monad $T^\prime $ such that there is a monad morphism $T \twoheadrightarrow T^\prime $ and the comparison $\A \to \Alg(T^\prime )$ is an equivalence.

We will proceed as follows:
\begin{enumerate}[(a)]
\item We will form an \emph{ordinary} left adjoint to $J$ by using Freyd's Adjoint Functor Theorem~\cite{maclane:cwm}.
\item We will show that the ordinary adjunction is enriched in $\Cat$ by using Proposition~\ref{prop:2-dim-gaft}.
\item We will construct the monad morphism $T \twoheadrightarrow T^\prime $ from the above adjunction and show the (enriched) equivalence $\A \simeq \Alg(T^\prime )$.
\end{enumerate}

Ad (a): We will show that $\A$ has and $J$ preserves ordinary limits. Since $J$ is fully faithful, it suffices to prove that $\A$ is closed in $\Alg(T)$ under ordinary limits. By assumption, $\A$ is closed in $\Alg(T)$ under 2-products. It is therefore closed under ordinary products as well, since 2-products and ordinary products coincide in $\Cat$. We need to show that it is closed also under equalisers. To this end, consider an equaliser diagram
$$
\xymatrix{
(A,a) \ar@{ >->}[r]
&
JX \ar@/^/[r]^{Js}
\ar@/_/[r]_{Jt}
&
JY
}
$$
in $\Alg(T)$. Equalisers in $\Alg(T)$ are computed on the level of underlying categories, which implies that $A \rightarrowtail UJX$ is faithful. Thus $(A,a)$ is a subalgebra of $JX$. Since the 2-category $\A$ is closed under subalgebras in $\Alg(T)$, we proved that it is closed under equalisers as well.

To establish the existence of a left adjoint for $J$, we now only need to find a solution set for every object $(A,a)$ of $\Alg(T)$. We claim that the solution set is the set $\{ h_i: (A,a) \twoheadrightarrow JX_i \mid i \in I \}$ of all the (representatives of the) quotients of $(A,a)$. This is indeed a set due to the nature of b.o.~fullness, recall Remark~\ref{rem:cowelpoweredness}. Given any morphism $h: (A,a) \to JX$, we can factorise it to obtain a triangle
$$
\xymatrix{
& (B,b) \ar@{ >->}[dd] \\
(A,a) \ar@{->>}[ru] \ar[rd]_{h} & \\
& JY
}
$$
and moreover, since $(B,b)$ is a subalgebra of $JY$, we have that $(B,b) \cong JX$ holds for some $X$ from $\A$, and the solution set condition is satisfied. The unit of the adjunction is constructed as follows: we take the product $\prod_{i \in I} JX_i$ of all the codomains of the quotients in the solution set, and factorise the mediating morphism $(h_i): (A,a) \to \prod_{i \in I} JX_i$ as in the following diagram.
$$
\xymatrix{
& JL(A,a) \ar@{ >->}[dd] \\
(A,a) \ar@{->>}[ru]^{\eta_{(A,a)}} \ar[rd]_{(h_i)} & \\
& \prod_{i \in I} JX_i
}
$$
Note that $\eta_{(A,a)}$ is b.o.~full for every algebra $(A,a)$.

Ad (b): We denote the adjunction from (a) by $L \dashv J$. To prove that this adjunction is enriched, we will use the argument contained in Remark~\ref{rem:2-dim-aft-usage}. Consider the solution set $\{ \eta_{(A,a)} \}$ for $(A,a)$. Since $\eta_{(A,a)}$ is b.o.~full for every $(A,a)$, it is the coequifier of its kernel.  The 2-dimensional universal property of the coequifier together with the fact that $J$ is fully faithful state that $\eta_{(A,a)}$ indeed satisfies the 2-dimensional solution set condition of Proposition~\ref{prop:2-dim-gaft} for every $(A,a)$. 

Ad (c): We can now define the 2-monad $T^\prime$ and the monad morphism $\phi: T \to T^\prime$ for which we will show the (enriched) equivalence $\A \simeq \Alg(T^\prime)$. Let us first settle the notation and write $(L \dashv J, \eta, \eps)$ for the adjunction $L \dashv J: \A \to \Alg(T)$, denote by $(F^T,U^T,\eta^T,\eps^T)$ the adjunction $F^T \dashv U^T: \Alg(T) \to \Cat$, and let $\mu^T : TT \to T$ be the multiplication of the 2-monad $T$.

This allows us to define the 2-functor $T^\prime  := U^T JL F^T$ which is the underlying endofunctor of the 2-monad $(T^\prime, \eta^{T^\prime}, \mu^{T^\prime})$ with the unit $\eta^{T^\prime}$ and the composition $\mu^{T^\prime}$ defined by the assignments
$$
\eta^{T^\prime} := U^T \eta F^T \cdot \eta^T, \qquad \mu^{T^\prime} := U^T J \eps L F^T \cdot U^T J L \eps^T J L F^T.
$$
Then there is a corresponding monad morphism $\phi = U^T\eta F^T: T \twoheadrightarrow T^\prime $. The proof that $\phi$ is indeed a monad morphism is standard and proceeds exactly as in the non-enriched case. Moreover, $\phi$ is a quotient, since
\begin{enumerate}
\item $\eta_A$ is a quotient for each algebra $(A,a)$, and
\item $U^T$ preserves quotients since $T$ does.
\end{enumerate}

Let us denote by
$$
\begin{tikzcd}
\A \ar[dotted]{rr}{K} \drar[swap]{U^T J} & & \Alg(T^\prime) \dlar{U^{T^\prime}} \\
& \Cat &
\end{tikzcd}
$$
the ordinary comparison functor. We will apply the ordinary Beck's theorem to infer that $K$ is an ordinary equivalence. Since $\A$ has and $U^T J$ preserves sifted colimits, $\A$ has and $U^T J$ preserves coequalisers of reflexive pairs. Moreover, since $U^T$ reflects isomorphisms and $J$ is fully faithful, the composite functor $U^T J$ also reflects isomorphisms. Therefore $K: \A \to \Alg(T^\prime)$ is indeed an ordinary equivalence functor.

We will now show that on objects, the inclusion $J: \A \to \Alg(T)$ factorises, up to isomorphism, as in the following triangle:
$$
\xymatrix{
& \Alg(T^\prime) \ar[dd]^{\Alg(\phi)} \\
\A \ar[ru]^{K} \ar[rd]_{J} & \\
& \Alg(T).
}
$$
Indeed, for any object $A$ of $\A$ the equality
$$
KA = (U^TJA, U^T J \eps_A \cdot U^T J L \eps^T_{JA})
$$
holds. The algebra $KA$ gets mapped by the functor $\Alg(\phi)$ to an algebra with a structure map
\begin{align*}
U^T J \eps_A \cdot U^T J L \eps^T_{JA} \cdot \phi_{U^T J A} & = U^T J \eps_A \cdot U^T J L \eps^T_{JA} \cdot U^T \eta_{F^T U^T J A} \\
& = U^T J \eps_A \cdot U^T \eta_{JA} \cdot U^T \eps^T_{JA} \\
& = U^T \eps^T_{JA},
\end{align*}
where the first equality holds by the definition of $\phi$, the second one follows from naturality of $\eta$, and the third one comes from the triangle identity of $L \dashv J$. But $(U^TJA, U^T \eps^T_{JA})$ is isomorphic to $JA$, as $(U^TJA, U^T \eps^T_{JA})$ is the image of $JA$ under the trivial comparison functor
$$
I: \Alg(T) \to \Alg(T).
$$
Both $J$ and $\Alg(\phi)$ are fully faithful in $\Cat$-enriched sense: the 2-functor $J$ is such by assumption and $\Alg(\phi)$ was proved to be fully faithful for a quotient monad morphism $\phi$ in the first part of the proof. We can conclude that the ordinary equivalence $K: \A \to \Alg(T^\prime)$ is enriched in $\Cat$, thus finishing the proof.
\end{proof}

\begin{remark}
A point that needs to be discussed is that we demand closure of $\A$ under sifted colimits in $\Alg(T)$ in the characterisation of equationally defined subcategories of $\Alg(T)$. It is true that in the original Birkhoff theorem there is no need to demand closure under any class of colimits whatsoever. However, even in the ordinary case of $\V = \Set$, closure under filtered (or directed) colimits is essential in the case of many-sorted universal algebra, see~\cite{arv-many-sorted-hsp}. In the case of $\V = \Cat$, at least the requirement for closure under filtered colimits is arguably expectable. The reason why our version of the Birkhoff theorem asks for an even stronger closure property, i.e., closure under \emph{sifted} colimits is the following. While finitary and strongly finitary monads on $\Set$ coincide (every finitary monad is strongly finitary), this is not the case for 2-monads on $\Cat$: a finitary 2-monad need not be strongly finitary. For example, the 2-monad $T$ that gives rise to the 2-category $\Alg(T)$ of categories $\C$ equipped with one "arrow-ary" operation $\C^\Two \to \C$ is finitary, but $T$ fails to preserve sifted colimits in general. Since we are dealing with \emph{strongly finitary} 2-monads on $\Cat$, the corresponding closure property is the closure under \emph{sifted} colimits.
\end{remark}

In the ordinary setting, full algebraic subcategories induced by a quotient monad morphism can be characterised as a special kind of orthogonal subcategories. Without substantial changes to the reasoning, the same characterisation can be obtained for the case of $\V = \Cat$, as it is shown in the following remark.

\begin{remark}
Given a 2-category $\X$ and a set $S = \{ f_i: X_i \to Y_i \mid i \in I \}$ of morphisms of $\X$, we will denote by $S^\perp$ the full subcategory $J: \Y \to \X$ spanned by the objects $Y$ that are orthogonal to all morphisms in $S$. The full equationally defined subcategories
$$
J: \A \to \Alg(T)
$$
of the 2-category $\Alg(T)$ of algebras for a strongly finitary 2-monad $T$ are precisely the orthogonal subcategories of $\Alg(T)$ of the form
$$
\A = \{ f: F^T n \twoheadrightarrow (C,c) \mid f \in I \}^\perp = I^\perp
$$
for some set $I$ of quotient morphisms in $\Alg(T)$. Moreover, each morphism in $I$ has as its domain a free algebra on a finite discrete category.

To see that one direction of this statement holds, observe that $\A$ is closed under subobjects in $\Alg(T)$: Given an algebra $(B,b)$ in $\A$ and its subalgebra $(A,a)$, we have for any $g: F^T n \to (A,a)$ a situation
$$
\begin{tikzcd}
F^T n \rar[->>]{f} \dar[swap]{g} & (C,c) \dar[dotted] \dlar[swap, dotted] \\
(A,a) \rar[>->, swap]{g} & (B,b)
\end{tikzcd}
$$
where the morphism $(C,c) \to (B,b)$ exists since $f \perp (B,b)$, and the diagonal exists by the diagonal property of the factorisation system. Closure of $\A$ under 2-products in $\Alg(T)$ follows easily from the universal property of 2-products. To show that $\A$ is closed under quotients in $\Alg(T)$, observe first that $\Alg(T)(F^T n, {-})$ preserves quotient maps. Indeed, we have that
$$
\Alg(T)(F^T n, {-}) \cong \Cat(n, U^T{-})
$$
holds, and both $U^T$ and $\Cat(n,{-})$ are easily seen to preserve quotient maps. This is equivalently saying that $F^T n$ is projective with respect to quotients. Then, given an algebra $(A,a)$, a quotient $h: (A,a) \twoheadrightarrow (B,b)$ and a morphism $g: F^T n \to (B,b)$, there exists a (not necessarily unique) factorisation
$$
\begin{tikzcd}
F^T n \dar[dotted, swap]{p} \drar{g} &  \\
(A,a) \rar[swap, ->>]{h} & (B,b)
\end{tikzcd}
$$
by projectivity of $F^T n$. Since $(A,a)$ is orthogonal to $f$, we obtain a triangle
$$
\begin{tikzcd}
F^T n \rar{f} \dar[swap]{p} & (C,c) \dlar[dotted]{o} \\
(A,a) &
\end{tikzcd}
$$
The composite $h \cdot o$ then proves that $f \perp (B,b)$. Indeed, given any other factorisation $g = i \cdot f$, it holds that $i = h \cdot o$ since $f$ is epi. The closure of $\A$ under sifted colimits follows from observing that $\Alg(T)(F^T n, {-})$ preserves sited colimits. This is the case since both $\Cat(n,{-})$ and $U^T$ preserve sifted colimits.

In the other direction, recall that reflective subcategories are \emph{always} orthogonality classes. In our case we have that
$$
\A = \{ \eta_{(A,a)}: (A,a) \twoheadrightarrow JL(A,a) \mid (A,a) \in \Alg(T) \}^\perp.
$$
We need to take a subset of the above class of morphisms such that the codomain of each morphism is a free algebra on a finite discrete category. For this, we use that every algebra $(A,a)$ is a sifted colimit of free algebras on finite discrete categories.
Now as the orthogonal lifting property is closed under colimits in the category of arrows, we get that
$$
\{ \eta_{(A,a)}: (A,a) \twoheadrightarrow JL(A,a) \mid (A,a) \in \Alg(T) \}^\perp
$$
is equal to the subcategory
$$
\{ \eta_{F^T n}: F^T n \twoheadrightarrow JLF^T n \mid n \in \Cat, n \mbox{ finite discrete} \}^\perp,
$$
as we needed.
\end{remark}

\section{Concluding remarks}

In this section we first discuss possible directions for future work. Then we conclude by showing that the kernel-quotient systems $\F_\bo$ and $\F_\so$ are inadequate for obtaining any kind of a strong Birkhoff-type theorem.

\subsection*{Equational logic for the $\F_\bof$ factorisation system}

In classical universal algebra, it is known that there is an \emph{equational logic} that is sound and complete with respect to the notion of \emph{equational consequence}. See Section~3.2.4 of~\cite{wechler} for a nice treatment. An obvious question is whether there is an ``equational logic'' sound and complete with respect to the notion of equational consequence that comes from our definition of what equation is in the  2-dimensional context. Finding such a calculus is a problem for future work.

\subsection*{Other factorisation systems}

We will discuss some problems concerning the factorisation systems (bijective on objects, fully faithful) and (surjective on objects, injective on objects and fully faithful) on $\Cat$.

We can see immediately that the situation is very different in the case of the (b.o., f.f.) factorisation system when compared to the (b.o.~full, faithful) case. Given a monad morphism $e: T \to T'$ with $e_X : TX \to T' X$ being b.o.\ for all categories $X$, the algebraic functor
$$
\Alg(e): \Alg(T') \to \Alg(T)
$$
need not be fully faithful in the 2-dimensional sense. This calls for a different approach to stating and proving a Birkhoff-style theorem. Indeed, trying to mimic the case of the (b.~.o.~full, faithful) factorisation system breaks down at the very beginning: we will not be able to characterise ``equational'' subcategories by their closure properties, as the subcategories need not be full. Even more goes wrong: not every b.o.\ functor is epimorphic in $\Cat$, and $\Cat$ is not cowellpowered with respect to b.o.~quotients.

Recall from Remark~\ref{rem:alg-meaning-ptwise} that a b.o.~full quotient $e: T \twoheadrightarrow T^\prime$ has the components $e_n: Tn \twoheadrightarrow T^\prime n$ pointwise b.o.~full in $\Cat$. Algebraically, this specifies new equations that have to hold between morphisms in a $\T^\prime$-algebra. However, no such algebraic meaning can be given in the case of a b.o.\ quotient $e: T \twoheadrightarrow T^\prime$. This is because the components $e_n: Tn \twoheadrightarrow T^\prime n$ are pointwise \emph{only b.o.} in $\Cat$. Thus $T^\prime$ as a monad may contain new 2-dimensional algebraic information, and in this context it does not make sense to talk about $\Alg(T^\prime)$ as of an ``equational'' subcategory of $\Alg(T)$.

The same remarks remain true when considering the (s.o., i.o.f.f.) factorisation system. Thus both the (b.o.~full, faithful) and (s.o., i.o.f.f.) factorisation systems would allow only for a very weak and generic Birkhoff theorem.

\end{document}